\documentclass[12pt,a4paper]{amsart}
\usepackage[margin=2.5cm]{geometry}
\usepackage{tikz,tkz-graph}
\usepackage{float}
\usetikzlibrary{calc}
\usepackage{pgfplots}
\pgfplotsset{every axis/.append style={
                    axis x line=middle,    
                    axis y line=middle,    
                    axis line style={-,color=blue}, 
                    xlabel={$x$},          
                    ylabel={$y$},          
            }}
\pgfplotsset{compat=1.13}

\usepackage[latin1]{inputenc}
\usepackage{amsmath}
\usepackage{amsthm}
\usepackage{amsfonts}
\usepackage{amssymb}
\usepackage{mathtools}
\usepackage{enumerate}
\usepackage{graphicx}
\usepackage{hyperref}
\usepackage{nicefrac}
\usepackage[all,cmtip]{xy}

\DeclareMathOperator{\Hom}{Hom}

\def\RHS{$\mathbb QH S^3$}

\def\cS{{\mathcal S}}
\def\V{{V}}

\def\cI{{\mathcal I}}
\def\ZZ{\mathbb{Z}}

\def\CC{\mathbb{C}}

\def\QQ{\mathbb{Q}}

\def\cO{\mathcal{O}}

\def\cF{\mathcal{F}}

\def\div{\textrm{div}}

\def\tt{\mathbf{t}}
\def\ZK{Z_K}

\newtheorem{thm}{Theorem}[section]  
\newtheorem{prop}{Proposition}[section]
\newtheorem{cor}{Corollary}[section]
\newtheorem{lemma}{Lemma}[section]
\newtheorem{def-lemma}{Definition-Lemma}[section]

\theoremstyle{remark}
\newtheorem{rem}{Remark}[section]

\theoremstyle{definition}
\newtheorem{dfn}{Definition}[section]
\newtheorem{exam}{Example}[section]

\makeatletter
\let\c@lemma\c@thm
\let\c@prop\c@thm
\let\c@propdef\c@thm
\let\c@proper\c@thm
\let\c@problem\c@thm
\let\c@conj\c@thm
\let\c@cor\c@thm
\let\c@rem\c@thm
\let\c@dfn\c@thm
\let\c@notation\c@thm
\let\c@exam\c@thm
\makeatother

\title[Delta invariant of curves on rational surfaces I. The analytic approach]
{Delta invariant of curves on rational surfaces I.\\ The analytic approach}

\author[J.I. Cogolludo]{Jos{\'e} Ignacio Cogolludo-Agust{\'i}n}
\address{Departamento de Matem\'aticas, IUMA\\
Universidad de Zaragoza\\
C.~Pedro Cerbuna 12\\
50009 Zaragoza, Spain}
\email{jicogo@unizar.es}

\author[T. L\'aszl\'o]{Tam\'as L\'aszl\'o}
\address{Alfr\'ed R\'enyi Institute of Mathematics,
Hungarian Academy of Sciences,
Re\'altanoda utca 13-15, H-1053, Budapest, Hungary}
\email{laszlo.tamas@renyi.hu}

\author[J. Mart\'in]{Jorge Mart\'in-Morales}
\address{Centro Universitario de la Defensa-IUMA\\
Academia General Militar\\
Ctra.~de Huesca s/n.\\
50090, Zaragoza, Spain}
\email{jorge@unizar.es}

\author[A. N\'emethi]{Andr\'as N\'emethi}
\address{Alfr\'ed R\'enyi Institute of Mathematics,
Hungarian Academy of Sciences,
Re\'altanoda utca 13-15, H-1053, Budapest, Hungary \newline
 \hspace*{4mm} ELTE - University of Budapest, Dept. of Geometry, Budapest, Hungary \newline \hspace*{4mm}
BCAM - Basque Center for Applied Math.,
Mazarredo, 14 E48009 Bilbao, Basque Country, Spain}
\email{nemethi.andras@renyi.hu }

\thanks{The first and third authors are partially supported by MTM2016-76868-C2-2-P and
Gobierno de Arag{\'o}n (Grupo de referencia ``{\'A}lgebra y Geometr{\'i}a'')
cofunded by Feder 2014-2020 ``Construyendo Europa desde Arag\'on''.
The third author is also partially supported by FQM-333 ``Junta de Andaluc{\'\i}a''. \\
The second and fourth authors are supported by NKFIH Grant ``\'Elvonal (Frontier)'' KKP 126683.
The second author was also supported by ERCEA Consolidator Grant 615655 - NMST, and partially
by the Basque Government through the BERC 2018-2021 program and Gobierno Vasco Grant IT1094-16,
by the Spanish Ministry of Science, Innovation and Universities: BCAM Severo Ochoa accreditation SEV-2017-0718.}

\begin{document}

\begin{abstract}
We prove that if $(C,0)$ is a reduced curve germ  on a rational surface singularity $(X,0)$ then
its delta invariant can be recovered by a concrete expression associated with
 the  embedded topological
type of the pair $C\subset X$. Furthermore, we also identify it with another (a priori)
embedded analytic invariant, which is motivated by the theory of adjoint ideals.
Finally, we connect our formulae with the local correction term
at singular points of the global Riemann--Roch formula, valid for projective normal surfaces, introduced by Blache.
\end{abstract}

\maketitle

\section{Introduction}
 The central object of the present paper is the germ of a reduced curve  on a
complex normal surface singularity. We wish to understand the behavior of crucial
invariants with respect to the analytic--topological comparison, and also
with respect to their role and local contributions
in the global geometry of Weil divisors on normal projective surfaces.

\subsection{}
First we discuss the local aspects. Let $(X,0)$ be a normal surface singularity  and
$(C,0)\subset (X,0)$ a reduced curve germ  on it. Let $r$  be the number of
irreducible components of $(C,0)$; this is the only topological invariant of the abstract curve germ
$(C,0)$.

Probably the most important numerical analytic invariant of the abstract curve $(C,0)$ is its delta invariant $\delta(C)$ (for definition
and several properties see e.g \cite{BG80,StevensThesis} or
section \ref{s:delta} here). Our guiding question is whether $\delta(C)$ can be read  from
the local embedded topological type of the pair $(C,0)\subset (X,0)$. For example, if
$(C,0)$ is Cartier, cut out by the local equation $f:(X,0)\to (\CC,0)$, then by \cite{BG80}
$2\delta(C)=r-1+\mu(f)$, where $\mu(f)$ is the Milnor number of $f$, which definitely can be determined from the embedded topological type (e.g. via A'Campo's formula \cite{AC},
or from the fact that the $\ZZ$--covering of $X\setminus C$ associated with the representation
$\pi_1(X\setminus C)\to \ZZ$  given by the Milnor fibration is homotopically the Milnor fiber).
However, if $C$ is not Cartier, then we cannot expect in general an embedded topological type
characterization of $\delta(C)$ (for a detailed discussion when
$(X,0)$ is a particular minimally elliptic singularity see  Example \ref{ex:NON}). Still,
one of the main results of the present note is that if we assume that $(X,0)$ is rational
then  such a characterization is possible.

Since this characterization is rather delicate, let us present it with more details.
For simplicity we will assume that the link $\Sigma$ of our normal surface singularity
is a rational homology sphere, denoted by \RHS\ (rational singularities satisfy this restriction).

We fix a good embedded resolution $\pi:\tilde X\to X$ of the pair $C\subset X$. As usual,
we consider the combinatorial package of the resolution (for details see section \ref{s:GenAn}):
$E=\pi^{-1}(0)$ is the exceptional curve, $\cup_vE_v$  is its decomposition into irreducible components, $L=H_2(\tilde X, \ZZ)=\ZZ\langle E_v\rangle_v$ is the lattice of $\pi$ endowed
with the negative definite intersection form $(E_v,E_w)_{v,w}$.
We identify the dual lattice $L'$ with those
rational cycles $\ell'\in L\otimes \QQ$ for which $(\ell',\ell)\in \ZZ$ for any $\ell\in L$.
Then it turns out that $L'/L$ is the finite group
$H_1(\partial \tilde X,\ZZ)$ ($\partial \tilde X=\Sigma$), which will be denoted by $H$.
Set $[\ell']$ for the class of $\ell'\in L'$ in $H$.

Let $K_\pi\in L'$ be the {\it canonical cycle}  of $\pi$, see (\ref{eq:KX}),
and we set for any $\ell'\in L'$ the
Riemann--Roch expression $\chi(\ell'):=-(\ell', \ell'+K_\pi)/2$: if $\ell\in L$ is effective then
$\chi(\ell)=h^0(\cO_\ell)-h^1(\cO_\ell)$.

Next, regarding $(C,0)$, we consider the strict transform $\widetilde{C}\subset \tilde X$. The embedded topological type is basically coded in the information that
 how many components of  $\widetilde{C}$ intersect each $E_v$, that is, in the intersection numbers
 $(\widetilde{C},E_v)_{\tilde X}$. Then we define the rational cycle $\ell'_C\in L'$ associated
 with $C$ having the property that $(\ell_C', E_v)+(\widetilde {C}, E_v)_{\tilde X}=0$ for every
 vertex  $v$.
(Hence, $\ell_C'+\widetilde {C}$, as a rational divisor, is numerically trivial; usually
it is denoted by $\pi^*(C)$, the {\it total transform of $C$},  a notation that will be used in the sequel.)

The first numerical embedded topological invariant we will consider is $\chi(-\ell'_C)$ (for several motivating examples see the body of the article). However, in our characterization we will need another,
a more subtle term, as well.

Let $\cS'$ be the Lipman (antinef) cone $\{\ell'\in L'\,:\, (\ell', E_v)\leq 0\ \mbox{for all $v$}\}$. By the negative definiteness of $(-,-)$ we know that $\cS'$ sits in the first quadrant
of $L\otimes {\mathbb R}$, and also
for any $h\in H$ there exists a unique $s_h\in \cS'$ such that $[s_h]=h$ and $s_h$ is minimal with these two properties. The cycle $s_h$ is zero only if $h=0$, and it usually is rather arithmetical and 
hard to find  explicitly. (For some concrete examples see~\cite{NemOSZ}.)

\begin{thm}\label{thm:INTR1}
If $(X,0)$ is rational then
$$\delta(C)=\chi(-\ell_C')-\chi(s_{[-K_\pi+\ell'_C]}).$$
In particular, $\delta(C)$ depends only on the embedded topological type of the
pair $(X,C)$.
\end{thm}
This  generalizes the main results of \cite{ji-JJ-numerical}, valid when  $(X,0)$ is a cyclic quotient.
The  message of the statement  is the same as the message of the
articles \cite{CDGZ-Moncurve,cdg,CDGZ-PScurves},
where the (analytic) semigroup of $C$ is compared with the Alexander polynomial of the embedding.
(The topological connections of the present manuscript with multivariable Poincar\'e series
will be treated in a forthcoming manuscript \cite{NEW}.)

In the proof of the statement we needed  as an intermediate step the following  `duality'  relation,
valid for any rational singularity and any $h\in H$:
\begin{equation}\label{eq:INTR1}
\chi(s_{[-K_{\pi}]+h})=\chi(s_{-h}).
\end{equation}

E.g., if $(C,0)$  is Cartier, then $[\ell'_C]=0$, and  $\chi(s_{[-K_\pi+\ell'_C]})=
\chi(s_{[-\ell'_C]})=\chi(s_0)=\chi(0)=0$, hence Theorem \ref{thm:INTR1} reads as
$\delta(C)=\chi(-\ell_C')$. In some sense, the difference
$\delta (C)-\chi(-\ell_C')$ measures non--triviality of the class of $C$ in
${\rm Weil}(X)/{\rm Cartier}(X)$ (we  will make this statement more precise
in \ref{ss:GLOBAL} below).

Though this identity (\ref{eq:INTR1})
 is topological in nature, it is the trace of the analytic/algebraic
Serre  duality. It  does not extend to any non--rational singularity,
cf. Example \ref{ex:DOESNOT}.

\subsection{}
The term $-K_\pi+\ell_C'$ suggests some relationship with adjoint  ideals, and
indeed, there exists another numerical embedded analytic invariant, motivated by the theory of
adjoints and constants of quasiadjunctions, which (by the next Theorem \ref{thm:INTR2})
equals the left (and the right)  hand side of Theorem
\ref{thm:INTR1}. The literature of adjoints is extensive, for applications in local singularity theory
one can consult e.g. several articles of Libgober (see for instance~\cite{Libgober-alexander,Libgober-characteristic})
and his school (e.g. \cite{ji-tesis}). Let us give a simple
example how one can produce such an invariant.
Consider the ordinary cusp $(C,0)$ given by the equation $f(x,y)=x^2-y^3=0$ in $(X,0)=(\CC^2,0)$.
The 2-form $w=\frac{dx \wedge dy}{f}$ can be pulled back to the (say, minimal)
embedded  resolution $\pi$ of $C\subset X$. We wish to find the ideal of germs $g$ such that the pullback of $g\omega$
has no pole along the exceptional curves. E.g., the local equation of $\pi^*w$ at the
intersection of the `last' exceptional divisor $\{u=0\}$ and the strict transform $\widetilde C=\{v=0\}$ is given by
$\frac{du \wedge dv}{u^2v}$. On the other hand, for any $g\in {\mathfrak m}_{X,0}$, the form $\pi^*(gw)$ has
poles only along $\widetilde C$. Hence  (by checking other points as well)
the wished ideal is ${\mathfrak m}_{X,0}$
with $\kappa_{X,0}(C):=\dim_\CC \cO_{X,0}/\mathfrak{m}_{X,0}=1$.
Note that $\delta(C)=1$ too, and the equality is not just a coincidence.

We  define this new invariant  (as a novelty of this note and  as a conceptual generalization of the
$\kappa$--invariant considered e.g. in \cite{ji-tesis,CM,ji-JJ-numerical})  via the
equivariant Hilbert series  $H(\tt)=\sum_{\ell'\in L'}{\mathfrak{h}}(\ell')\tt^{\ell'}$ of $(X,0)$.
Here ${\mathfrak{h}}(\ell')$ stays as the codimension in the local algebra of the universal abelian cover
of the ideal associated with $\ell'$ by the equivariant divisorial filtration of $\pi$.
It is one of the strongest analytic invariants of $(X,0)$. (For more see \ref{s:Pseries}.)
Now, having the  Weil divisor $C$, we define the kappa--invariant of $C\subset X$ by
$$\kappa_X(C):=\mathfrak{h}(-K_\pi+\ell'_C).$$
 \begin{thm} \label{thm:INTR2}
 If $(X,0)$ is rational and $(C,0)$ is a reduced curve on it  then
$$\delta(C)=\kappa_X(C).$$
 \end{thm}
\subsection{}
 It is instructive to consider the following table associated with a pair $(X,C)$:

\vspace{3mm}

\begin{center}
\begin{tabular}{|c | c | c |} \hline
  &$\  \mbox{abstract invariants of $C$} \ $ &
  $\ \mbox{embedded invariants of $C\subset X$} \ $ \\ \hline
  $\ \mbox{topological } \ $ & $r$ & $\ \chi(-\ell_C'), \ \chi(s_{[-K_\pi+\ell_C']}), \
  \chi(s_{-[\ell'_C]})\ $ \\ \hline
   $\ \mbox{analytical } \ $ & $\delta(C)$ & $\ \kappa_X(C) \ $ \\ \hline
\end{tabular}
\end{center}

\vspace{3mm}

From above,
for rational $(X,0)$ we have
$\delta(C)=\kappa_X(C)=\chi(-\ell_C')-\chi(s_{-[\ell_C']})$.

\subsection{} Theorems \ref{thm:INTR1} and \ref{thm:INTR2} are immediate consequences of the next theorem, in which some of the statements are valid for non--rational germs too, and we also
emphasize the peculiar properties which should be additionally proved  in the rational case.

\begin{thm}\label{thm:INTR3}
Let $(C,0)$ be a reduced curve germ  on $(X,0)$, $\pi^*C=\widetilde{C}+\ell'_C$ and $h:=[\ell'_C]$.
\begin{enumerate}
 \item 
 If $X$ is a normal surface singularity with \RHS\ link, then
\begin{equation*}
\chi(-\ell'_C)-\chi(s_{-h})=h^1({\tilde X}, \cO_{{\tilde X}}(\ell'_C))-
h^1({\tilde X}, \cO_{{\tilde X}}(-s_{-h}));\end{equation*}
\begin{equation*}
\kappa_X(C)=\chi(-\ell_C')-\chi(s_{[-K_\pi]+h})+h^1(\cO_{\tilde X}(-s_{[-K_\pi]+h})).
\end{equation*}
 \item 
 If $X$ is a normal surface singularity (with a non--necessarily \RHS\ link), then
\begin{equation*}
\delta(C)=h^1({\tilde X}, \cO_{{\tilde X}}(-\widetilde{C}))-p_g(X).
\end{equation*}
 \item 
 If $X$ is a rational singularity, then
\begin{enumerate}
 \item
 $p_g(X)=0$, $h^1({\tilde X}, \cO_{{\tilde X}}(-s_{-h}))=0$,
 $h^1(\cO_{\tilde X}(-s_{[-K_\pi]+h}))=0$,
 \item
 $\chi(s_{-h})= \chi(s_{[Z_K]+h})$,
 \item
 $h^1({\tilde X}, \cO_{{\tilde X}}(\ell'_C))=h^1({\tilde X}, \cO_{{\tilde X}}(-\widetilde{C}))$.
\end{enumerate}
\end{enumerate}
\end{thm}

The proof uses generalized Laufer computation sequences (from \cite{Laufer-rational,NemOSZ,Nem-GR,Nem-PS}), and vanishing theorems: the Grauert--Riemenschneider
vanishing \cite{GrRie,Laufer-rational,Ram}, its generalization, the local version of the general vanishing from  \cite{EFBook},
both valid for arbitrary surface singularities. Then, we use for rational singularities
Lipman's vanishing \cite{Lipman} (all of them will be reviewed in \ref{ss:vanishingth}).
Additionally we need to prove a new
vanishing theorem, valid for rational surface singularities, namely
$$h^1(\cO_{\tilde X}(K_{\pi}+s_h))=0 \ \ \ \mbox{for any $h\in H$}.$$
(Here $K_{\pi}+s_h$ in general is not an integral cycle, for the definition of the `natural line bundles' $\cO_{\tilde X}(\ell')$ for $\ell'\in L'$ see \ref{s:Pseries}.)

\subsection{}\label{ss:GLOBAL} In the literature there are several articles targeting the generalized
Riemann--Roch Theorem and adjunction formulae for Weil divisors on projective algebraic normal surfaces.
The formulae focus on the correction terms given by the local contributions of the local singular
points of the surface. Here we will follow Blache's approach \cite{Blache-RiemannRoch}, which was also a
motivation for us. (Below we use the standard notations of algebraic geometry.)

\begin{thm}[\cite{Blache-RiemannRoch}]  \label{thm:INTR4}
For every algebraic normal surface germ $(X,0)$ there exists a unique map
$A_{X,0}:{\rm Weil}(X,0)/{\rm Cartier}(X,0) \to \QQ$ with
\begin{enumerate}
\item $A_{X,0}(-D)=A_{X,0}(-K_X+D)$ for any Weil divisor $D$,
\item $A_{X,0}(C)= \chi(-\ell_C')-\delta(C)$ for any reduced curve $(C,0)\subset (X,0)$,
\end{enumerate}
such that for every projective normal surface $Y$ and every Weil divisor $D$ of $Y$ and every reduced curve $C\subset Y$ one has
\begin{equation*}\begin{split}
\chi(\cO_Y(D))=\chi(\cO_Y)+ \frac{1}{2} ( D, D-K_Y)-\sum_{y\in {\rm Sing}(Y)}
A_{Y,y}(-D),\\
C^2+(C,K_Y)=2p_a(C)-2+2 \cdot \sum_{y\in {\rm Sing}(Y)}
A_{Y,y}(C).\end{split}\end{equation*}
\end{thm}
This combined with our main result gives the following.
\begin{cor} If $(X,0)$ is rational then $A_{X,0}(C)= \chi(s_{[-K_\pi+\ell'_C]})$. This is an embedded topological characterization of  Blache's correction term $A_{X,0}$.
\end{cor}
Note also that the identity (\ref{eq:INTR1}) is in a perfect compatibility with
Theorem \ref{thm:INTR4}{\it (1)}.
\subsection{}

The paper is organized as follows: in section~\ref{s:GenAn} a brief description of the basic tools to study \RHS\ surface singularities is given. In section~\ref{sec:kappa-delta},  after some motivating examples,
the kappa invariant is defined. In section~\ref{sec:mainthm}
we prove the main result, Theorem \ref{thm:INTR3}. In the body of the article we list several examples. In section~\ref{s:delta} further examples and applications are given. For instance
we exhibit the limits of our identities in the non--rational cases, we exemplify
rational Kulikov singularities, and we also compare (via Blache's invariant)
 Mumford and Hironaka's intersection multiplicities, valid for  curve germs.

\subsection*{Acknowledgments}
The first and third authors want to thank the Fulbright Program (within the Jos\'e Castillejo and Salvador de Madariaga
grants by Ministerio de Educaci\'on, Cultura y Deporte) for their financial support while writing this paper. They also want
to thank the University of Illinois at Chicago, especially Anatoly Libgober, Lawrence Ein, and Kevin Tucker for their warm welcome
 and support in hosting them as well as their useful discussions.

\section{Preliminaries}\label{s:GenAn}

\subsection{Lattices associated with a resolution}
\label{sec:lattice}
Let us consider a complex normal surface singularity $(X,0)$. Let $\pi:\tilde X\to X$ be a good resolution with dual
resolution graph $\Gamma$ whose set of vertices are denoted by $\V$. Let $\{E_v\}_{v\in \V}$ be the irreducible components
of the exceptional set $\pi^{-1}(0)$. We assume that the link $\Sigma$ is a rational homology sphere, i.e.
$\Gamma$ is a connected tree and all $E_v$ are rational. (For more regarding this section see~\cite{NemOSZ,Nem-PS,Nem-CLB,Okuma-abelian}.)

Define the lattice $L$ as  $H_2(\tilde X,\ZZ)$, it is generated by the exceptional divisors $E_v$, $v\in \V$, that is,
$L=\oplus_{v\in \V} \ZZ\langle E_v \rangle$. In the homology exact sequence of the pair $(\tilde X, \Sigma)$ ($\partial \tilde X=\Sigma$) one has
$H_2(\Sigma,\ZZ)=0$, $H_1(\tilde X, \ZZ)=0$,  hence the exact sequence  has the form:
\begin{equation}
\label{eq:ses}
\array{ccccccccc}
0 & \to & L & \to & H_2(\tilde X,\Sigma,\ZZ) & \to & H_1(\Sigma,\ZZ) & \to & 0.\\
\endarray
\end{equation}
Set  $L':= \Hom(H_2(\tilde X,\ZZ),\ZZ)$.
The Lefschetz--Poincar\'e duality $H_2(\tilde X,\Sigma,\ZZ)\cong H^2(\tilde X,\ZZ)$
defines a perfect  pairing
$L\otimes H_2(\tilde X,\Sigma,\ZZ)\to \ZZ$.  Hence $L'$
can be identified with $H_2(\tilde{X}, \Sigma, \ZZ)$. By~\eqref{eq:ses}
$L'/L\cong H_1(\Sigma,\ZZ)$, which will be denoted by $H$.
Since $H$ is finite, one has the embedding $L'\subset L_{{\mathbb Q}}:=
 L\otimes {\mathbb Q}$ too,
and $L'$ identifies with the rational cycles $\{\ell'\in L_{{\mathbb Q}}\,:\, (\ell',L)_{{\mathbb Q}}\in \ZZ\}$, where
$(\,,\,)$ denotes the intersection form on $L$  and $(\,,\,)_{{\mathbb Q}}$  its extension to $L_{{\mathbb Q}}$.
 Hence, in the sequel we regard $L'$ as $\oplus_{v\in \V} \ZZ\langle E^*_v \rangle$,
 the lattice generated by the rational cycles  $E^*_v\in L_{{\QQ}}$,
$v\in \V$,  where $(E_u^*,E_v)_{{\QQ}}=-\delta_{u,v}$ (Kronecker delta) for any $u,v\in \V$.
 The inclusion $L\subset L'$  in the bases $\{E_v\}_v$ and $\{E_v^*\}_v$  is given by
$-M$, where $M$ is the intersection matrix of $L$, that is,
$E_v=-\sum_{u\in \V} (E_v,E_u) E_u^*$.

Since $M$ is negative definite the matrix  $-M^{-1}$
has  positive  entries, and the   $E^*_v$'s   are the columns of  $-M^{-1}$.
The elements $E^*_v$ have the following
geometrical interpretation as well: consider $\gamma_v\subset \tilde X$ a curvette associated with $E_v$, that is, a smooth
irreducible curve in $\tilde X$ intersecting $E_v$ transversally. Then
$\pi^*\pi_*\gamma_v=\gamma_v+E_v^*$.

Let $K_{\tilde X}$ be the canonical divisor of  the smooth surface $\tilde X$. The canonical divisor in $X$ is defined as
$K_X:=\pi_*K_{\tilde X}$. In particular,  $K_\pi:=K_{\tilde X}-\pi^*K_X$ is a rational cycle supported on the exceptional
set $\pi^{-1}(0)$; it is called the  {\it canonical cycle}
 of $\pi$, and it is determined topologically/numerically  by the
linear system of {\it adjunction relations}
\begin{equation}
\label{eq:KX}
(K_\pi+E_v,E_v)+2=0, \textrm{ for all } v\in \V.
\end{equation}
In particular, $K_\pi\in L'$.
Sometimes it is more convenient to use the (anti)canonical cycle $\ZK:=-K_\pi$.
Using~\eqref{eq:KX}, $Z_K$ can be written as
\begin{equation}\label{eq:ZK}
\ZK=E-\sum_{v\in \V} (2-\text{val}(v)) E^*_v.
\end{equation}
where $E=\sum_{v\in \V}E_v$ and $\text{val}(v)$ is the valency of $v$ in $\Gamma$.

For any $\ell'\in L'$ we write $\chi(\ell'):=-(\ell', \ell'-Z_K)/2$. By the Riemann--Roch theorem, for any
effective $\ell\in L_{>0}$ one has $\chi(\ell)=h^0({\mathcal O}_\ell)-h^1({\mathcal O}_\ell)$.

\subsection{\texorpdfstring{$H$}{H}--representatives and the Lipman cone}\label{ss:HrepLip}
For $\ell'_1,\ell'_2\in L_\QQ$ with $\ell'_i=\sum_v l'_{iv}E_v$ ($i=\{1,2\}$) one considers a partial  order
relation $\ell'_1\geq \ell'_2$ defined coordinatewise by $l'_{1v}\geq l'_{2v}$ for all $v\in\V$.
In particular, $\ell'$ is an effective rational cycle if $\ell'\geq 0$, written as $\ell'\in L'_{\geq 0}$.
We set also $\min\{\ell'_1,\ell'_2\}:= \sum_v\min\{l'_{1v},l'_{2v}\}E_v$.

Given an element $\ell'\in L'$ we denote by $[\ell']\in H$ its class in $H=L'/L$.
The lattice $L'$ admits a partition parametrized by the group $H$,  where for any $h\in H$ one sets
\begin{equation}
\label{eq:Lprime}
L'_h=\{\ell'\in L'\mid [\ell']=h\}\subset L'.
\end{equation}
Note that $L'_0=L$.
Given an $h\in H$ one can define $r_h:=\sum_v l'_v E_v\in L'_h$ as
the unique  element of $L_h'$ such that $0\leq l'_v<1$.
Equivalently,  $r_h=\sum_v \{l'_v\} E_v$ for any $\ell'=\sum_v l'_v E_v\in L'_h$, where $0\leq \{\cdot\}<1$
represents the fractional part.

We define the rational  Lipman cone by
$$\cS_\QQ:=\{\ell'\in L_\QQ \ | \ (\ell',E_v)\leq 0 \ \mbox{for all} \ v\in \V\},$$
which is a cone generated over $\QQ_{\geq 0}$ by $E^*_v$.
Define $\cS':=\cS_\QQ\cap L'$ as the semigroup (monoid) of anti--nef rational cycles of $L'$;  it
 is generated over $\mathbb{Z}_{\geq 0}$ by the cycles $E^*_v$.

In particular, if $(C,0)\subset (X,0)$ is a reduced curve (or only a nonzero effective Weil
divisor), and we write $\pi^*(C)=\widetilde{C}+\ell_C'$ (as in the introduction) with
$\ell_C'\in L'$, then necessarily $\ell_C'\in \cS'\setminus \{0\}$, in particular, $\ell_C'$
 is nonzero effective.

The Lipman cone $\cS'$  also admits a natural equivariant partition indexed  by $H$ by $\cS'_{h}=\cS'\cap L'_h$.
Note the following properties of the Lipman cone:
\begin{enumerate}
\item[(a)] $s_1,s_2\in \cS'_{h}$ implies $s_2-s_1\in L$ and $\min\{s_1,s_2\}\in \cS'_h$,
\item[(b)]\label{prop:sh}
for any $s\in L_{{\QQ}}$, the set $\{s'\in \cS'_h \mid s'\not\geq  s\}$ is finite, since $-M^{-1}$
has positive  entries.
\item[(c)]  for any $h$ there exists a unique  \emph{minimal cycle} $s_h:=\min \{\cS'_{h}\}$ (see~\ref{ss:GLA} below).
\item[(d)] $\cS'_h$ is a cone with \emph{vertex} $s_h$ in the sense that $\cS'_h=s_h+\cS'_0$.
\end{enumerate}
Following~\cite{NemOSZ} in the next subsection we  describe a generalization of Laufer's algorithm (see~\cite{Laufer-rational})
that can be used to calculate $s_h$.
\subsubsection{{\bf Generalized Laufer's algorithm}}\label{ss:GLA}
\cite[Lemma 7.4]{NemOSZ} For any $\ell'\in L'$ there exists a unique
minimal element $s(\ell')$  of the set $\{s\in \cS' \,:\, s-\ell'\in L_{\geq 0}\}$. It  can be obtained by the following algorithm.
Set $x_0:=\ell'$. Then one constructs  a computation sequence $\{x_i\}_i$ as follows.
If $x_i$ is already constructed and $x_i\not\in\cS'$ then there exists some $E_{u_i}$ such that $(x_i,E_{u_i})>0$.
Then take $x_{i+1}:=x_i+E_{u_i}$ (for some choice of $E_{u_i}$). Then the procedure after finitely many steps stops,
say at $x_t$, and necessarily $x_t=s(\ell')$.

Note that $s(r_h)=s_h$ and $r_h\leq s_h$, however, in general $r_h\neq s_h$.
(This fact does not contradict the minimality of $s_h$ in $\cS'_h$ since $r_h$ might not sit in $\cS'_h$.)
Also, if $\ell'\in L'_{\leq 0}$, then $s(\ell')=s_{[\ell']}$.

\subsection{Local divisor class group}\label{ss:LDCG} Using the exponential exact sequence of $\tilde X$ (and the notation $H^1(\tilde X, \cO^*_{{\tilde X}})=
{\rm Pic}(\tilde X)$ and,  as usual,  $L'=H^2(\tilde X,\ZZ)\simeq H_2(\tilde X,\Sigma,\ZZ)$) we get
\begin{equation}\label{eq:picard} 0\to H^1(\tilde X, \cO_{{\tilde X}})\to {\rm Pic}(\tilde  X)\stackrel{c_1}
{\longrightarrow} L'\to 0 \ \ \ \ \mbox{($c_1$=first Chern class)}.\end{equation}
Note that $L$ embeds naturally both in $L'$ and in ${\rm Pic}(\tilde  X)$ (in the second one by $\ell \mapsto \cO_{{\tilde X}}(\ell)$).
The group ${\rm Pic}(\tilde  X)/L$ is the {\it local divisor class group of }~$X$, that is, the group of local Weil divisors modulo the local Cartier divisors.
In particular, we have (the resolution independent) exact sequence
\begin{equation}\label{eq:picard2}
0\to H^1(\tilde X, \cO_{{\tilde X}})\to {\rm Weil}(X)/ {\rm Cartier}(X)\to H_1(\Sigma, \ZZ)\to 0.
\end{equation}
Recall that $p_g(X)=h^1(\tilde X, \cO_{{\tilde X}})$ is the {\it geometric genus}~of the germ
$(X,0)$.  The germ $(X,0)$ is called
{\it rational}~if $p_g(X)=0$.  By the above exact sequences, for rational singularities one has ${\rm Pic}(\tilde X)=L'$ (that is,
any line bundle of $\tilde X$ is determined topologically by its first Chern class) and also, the local divisor class group is
isomorphic with $H=H_1(\Sigma, \ZZ)$. (The morphism is induced as follows: take a divisor $D$, then the homology class of its oriented boundary
$\partial D\subset \partial  \tilde X=\Sigma$  gives the correspondence. However, here a warning is appropriate:
if $C$ is a reduced Weil divisor germ in $X$, and we set $\pi^*C=\widetilde{C}+\ell'_C$, where $\widetilde{C}$ is the
strict transform and $\ell'_C\in L'$, then usually in this manuscript we set $h=[\ell'_C]$. Hence, since
$\widetilde{C}+\ell'_C=0$ in $H_2(\tilde X, \Sigma,\ZZ)$, the class of $\partial \widetilde{C}$ is $-h$.)

\subsection{The Hilbert series of surface singularities}\label{s:Pseries}
\label{ss:uac}
We fix a good resolution $\pi$ of $(X,0)$.
 Consider $c:(Y,0)\to (X,0)$, the universal abelian cover of $(X,0)$, let
$\tilde Y$ be the normalized pull--back of $\pi$ and $c$, and denote by $\pi_Y$ and $\widetilde{c}$
the induced maps by the pull--back completing the following commutative diagram.
\begin{equation}
\label{eq:diagram}
\xymatrix{
\ar @{} [dr] | {\ }
\tilde{Y} \ar[r]^{\widetilde{c}} \ar[d]_{\pi_Y} & \tilde{X} \ar[d]^{\pi} \\
Y \ar[r]_{c} & X
}
\end{equation}

We define the following $H$--equivariant $L'$--indexed divisorial filtration on the local ring $\mathcal{O}_{Y,0}$
for any given $\ell'\in L'$:
\begin{equation}
\label{eq:Ffiltration}
\mathcal{F}(\ell'):=\{g\in \cO_{Y,0} \mid \div (g\circ \pi_Y)\geq \tilde c^* (\ell')\},
\end{equation}
where the pull--back $\tilde c^*(\ell')$ is an integral cycle in $\widetilde Y$ for any $\ell'\in L'$,
cf.~\cite[Lemma 3.3]{Nem-CLB}. The natural action of $H$ on $(Y,0)$ induces an action on $\cO_{Y,0}$ as follows:
$h\cdot g(y)=g(h\cdot y)$, $g\in \cO_{Y,0}$, $h\in H$. This action decomposes
$\cO_{Y,0}$ as $\oplus_{\lambda\in \hat{H}} (\cO_{Y,0})_{\lambda}$ according to the characters
$\lambda \in \hat{H}:={\rm Hom}(H,\CC^*)$, where
\begin{equation}
\label{eq:Heigenspaces}
(\cO_{Y,0})_{\lambda}:=\{g\in \cO_{Y,0} \mid g(h\cdot y)=\lambda (h)g(y),\ \forall y\in Y, h\in H\}.
\end{equation}
Note that there exists a natural isomorphism $\theta:H\to \hat{H}$ given by
$h\mapsto \exp(2\pi \sqrt{-1} (\ell',\cdot ))\in \Hom(H,\CC^*)$, where  $\ell' $ is any element of $L'$ with
$h=[\ell']$. In order to simplify our notations we write
 $(\cO_{Y,0})_h$ for $(\cO_{Y,0})_{\theta(h)}$ (and similarly for any
linear $H$--representation).

The subspace  $\mathcal{F}(\ell')$  is invariant under this
action and $ \mathcal{F}(\ell')_h=\mathcal{F}(\ell')\cap (\cO_{Y,0})_h$. Thus, one can define the {\em Hilbert function}
 $\mathfrak{h}(\ell')$ for any $\ell'\in L'$
as the dimension of the $\theta([\ell'])$--eigenspace $(\cO_{Y,0}/\mathcal{F}(\ell'))_{[\ell']}$.
The corresponding multivariable  \emph{Hilbert series} is
\begin{equation}
\label{eq:Hilbert}
H(\mathbf{t})=\sum_{\ell'\in L'} \mathfrak{h}(\ell')\mathbf{t}^{\ell'}\in \mathbb{Z}[[L']],
\end{equation}
where $\mathbb{Z}[[L']]$ is the $\ZZ$--module of formal series in the monomials
$\mathbf{t}^{\ell'}:=\prod_{v\in \V}t_v^{l'_v}$ for any $\ell'=\sum_{v}l'_v E_v$, where each
$l'_v$ might be rational from  $\frac{1}{d}\ZZ$ with~$d=|H|$.

The $H$--eigenspace decomposition of $\tilde c_{*}(\cO_{\widetilde Y})$ is given by (see~\cite{Nem-PS,Okuma-abelian})
$$\tilde c_{*}(\cO_{\widetilde Y})=\bigoplus_{h\in H}\cO_{\widetilde X}(-r_h) \   \ \mbox{with} \
\cO_{\widetilde X}(-r_h)=(\tilde c_{*}(\cO_{\widetilde Y}))_h,$$
where $\cO_{\widetilde X}(\ell')$ is the unique  line bundle $\mathcal L$ on $\widetilde X$ satisfying
$\tilde c^*\mathcal L=\cO_{\widetilde Y}(\tilde c^*(\ell'))$ (see~\cite[3.5]{Nem-CLB}) and $r_h$ is the
representative of $h$ as in section~\ref{ss:HrepLip}.
As a word of caution, note that $r_h$ is a $\QQ$--divisor in $\widetilde X$, however the notation
$\cO_{\widetilde{X}}(-r_h)$ here is different from the one used  by Sakai in~\cite{Sakai84}.
In particular, $\cO_{\widetilde{X}}(-r_h)$ and $\cO_{\widetilde{X}}(\lfloor -r_h \rfloor)$ denote different objects.

This provides the following alternative expression of the Hilbert function (cf.~\cite[Corollary.~4.2.4]{Nem-PS})
\begin{equation}\label{eq:hdim}
\mathfrak{h}(\ell')=\dim \frac{H^0(\widetilde X, \cO_{\widetilde X}(-r_h))}{H^0(\widetilde X, \cO_{\widetilde X}(-r_h-\ell))}
\end{equation}
for any $\ell'=\ell+r_h>0$.

\subsection{{\bf Some useful exact sequences and vanishing theorems}}\label{ss:vanishingth}
For  $\ell\in L_{>0}$ and $\ell'_1\in L'$ consider the
cohomology exact sequence associated with the exact sequence
\begin{equation}
\label{eq:restriction2}
0\to \cO_{\widetilde X}(-\ell-\ell'_1)\to \cO_{\widetilde X}(-\ell'_1)\to \cO_{\ell}(-\ell'_1)\to 0.
\end{equation}
Applying~\eqref{eq:restriction2} to $\ell'_1=r_h$ and writing $\ell'=\ell+r_h$ one obtains
$$\mathfrak{h}(\ell') -\chi(\cO_{\ell}(-r_h))+h^1(\cO_{{\tilde X}}(-\ell'))-h^1(\cO_{{\tilde X}}(-r_h))=0$$
or equivalently,
\begin{equation}\label{eq:hell}
\mathfrak{h}(\ell') =\chi(\ell')-h^1(\cO_{{\tilde X}}(-\ell'))-\chi(r_h)+ h^1(\cO_{{\tilde X}}(-r_h)).
\end{equation}
Similarly, in the  generalized Laufer's algorithm, when  $x_{i+1}=x_i+E_{u_i}$
and  $(x_i,E_{u_i})>0$ (see~\ref{ss:GLA}), the choice $\ell_1'=x_i$ and $\ell=E_{u_i}$ in~\eqref{eq:restriction2}   applied  repeatedly gives
\begin{equation}\label{eq:sell'}
h^1(\cO_{{\tilde X}}(-\ell'))-\chi(\ell')=h^1(\cO_{{\tilde X}}(-s(\ell')))-\chi(s(\ell')).
\end{equation}
Note that, if $\ell'=r_h$ then $s(r_h)=s_h$, hence
\begin{equation}\label{eq:rhsh}
h^1(\cO_{{\tilde X}}(-r_h))-\chi(r_h)=h^1(\cO_{{\tilde X}}(-s_h))-\chi(s_h).
\end{equation}

Regarding the cohomology group $h^1({\tilde X}, {\mathcal L})$ there are several
 useful vanishing theorems.

\begin{thm}[\textbf{Grauert--Riemenschneider vanishing~\cite{GrRie,Laufer-rational,Ram}}]
\label{thm:GR}
For any $(X,0)$ (even if $\Sigma$ is not a \RHS) and
 any $\mathcal L\in {\rm Pic}(\tilde X)$ with
$-c_1({\mathcal L})\in Z_K+ \cS'$ one has $h^1({\tilde X},{\mathcal L})=0$.
\end{thm}

\begin{thm}[\textbf{Generalized Grauert--Riemenschneider vanishing~\cite{EFBook}}]
\label{thm:GGR}  For any $(X,0)$ (even if $\Sigma$ is not a \RHS),
 for any $\mathcal L\in {\rm Pic}(\tilde X)$ and for any $\Delta \in L_{\QQ}$ with
 $\lfloor \Delta \rfloor=0$, if
$-c_1({\mathcal L})\in -\Delta + Z_K+ \cS_{\QQ}$, then $h^1({\tilde X},{\mathcal L})=0$.
\end{thm}

\begin{thm}[\textbf{Lipman's vanishing~\cite[Theorem 11.1]{Lipman}}]
\label{thm:Lipman}
If $X$ is a rational singularity and
$\mathcal L\in {\rm Pic}(\tilde X)$ with
$-c_1({\mathcal L})\in \cS'$, then $h^1({\tilde X},{\mathcal L})=0$.
\end{thm}

\section{The kappa and delta invariants of reduced Weil divisors}
\label{sec:kappa-delta}
Consider $(X,0)$ with \RHS\ link,
and $(C,0)$ a reduced curve germ on it. As a standard notation for the
rest of the paper, consider $\pi:\tilde X \to X$ a good resolution of $(X,C)$ and write the total transform of
$C$ by $\pi$ as $\pi^*C=\ell'_C+\widetilde{C}$, with $\ell'_C\in L'_{\geq 0}$.
Since $C$ is not necessarily Cartier, $h:=[\ell'_C]\in H$ is not necessarily zero.

In order to motivate the general definition of the kappa invariant associated with $C$ we will study some
particular cases. In these cases, one can also see the (expected) connection with certain embedded topological
invariant (namely with $\chi(-\ell'_C)$) as well as with the delta invariant $\delta(C)$ of the
abstract curve germ $C$.
 (For more on  $\delta(C)$ see section~\ref{s:delta}.)

\subsection{Plane curve singularities}\label{ss:planecurve}
Take $f\in\cO_{\CC^2,0}$, which defines an isolated singularity, and it has $r$ local irreducible
components. Let $\pi:\tilde X\to \CC^2$ be the minimal good embedded resolution of the pair $C:=\{f=0\}\subset \CC^2$.
Let us define the ideal $\cI_C$ as the set of germs $g\in\cO_{\CC^2,0}$ such that
$\pi^*(g\cdot \frac{dx\wedge dy}{f})$ has a regular extension over all $\tilde X$ (except the strict transform
$\widetilde{C}$ of $C$). We define $\kappa(C):=\dim (\cO_{\CC^2,0}/\cI_C)$.
One of our goals is to give several interpretations of~$\kappa(C)$.

Write $\pi^*C$ as $\ell_C+\widetilde{C}$, with $\ell_C\in L$. Then one verifies that for the minimal good resolution
$\ell_C+Z_K$ is effective (usually $Z_K$ is not). Since the divisor along $E$ of $\pi^*(dx \wedge dy)$ is $-Z_K$, we obtain that
$\cI_C=H^0(\tilde X, \cO_{{\tilde X}}(-\ell_C-Z_K))\subset H^0(\tilde X, \cO_{{\tilde X}})=\cO_{\CC^2,0}$.
In the cohomology  exact sequence of
$$0\to  \cO_{{\tilde X}}(-\ell_C-Z_K)\to \cO_{{\tilde X}}\to   \cO_{\ell_C+Z_K}\to 0,$$
 $H^1(\tilde X, \cO_{{\tilde X}}(-\ell_C-Z_K))=0$ by Grauert--Riemenschneider vanishing~\ref{ss:vanishingth},
and $H^1(\tilde X, \cO_{{\tilde X}})=0$ since $\CC^2$ is smooth with geometric genus zero. Hence
$H^1( \cO_{\ell_C+Z_K})=0$ too.  Therefore, $\kappa(C)=\chi(\cO_{{\ell_C+Z_K}})=\chi(\ell_C+Z_K)=\chi(-\ell_C)$.

Next, using~\eqref{eq:ZK} and the fact that $\ell_C+\widetilde{C}
$ is numerically trivial, and the well--known formula of A'Campo for
the Milnor number $\mu(f)$ of $f$~\cite{AC}, a computation gives $\chi(-\ell_C)=(r-1+\mu(f))/2$.
On the other hand, by Milnor's formula~\cite{MBook}, we have that $(r-1+\mu(f))/2$
equals the delta invariant $\delta(C)$. Hence
$$\kappa(C)=\chi(-\ell_C)=\delta(C).$$

\subsection{\texorpdfstring{$\kappa$}{kappa}--invariant for cyclic quotient singularities}
\label{ss:kappaCYCLIC}
Let $(X,0)$ be the cyclic quotient singularity $\CC^2/\ZZ_d$, sometimes denoted also as $\frac{1}{d}(1,q)$, according to the action
$\ZZ_d\times  \CC^2\to \CC^2$, $\xi*(x,y)=(\xi x,\xi^q y)$, where $\gcd(d,q)=1$ and $\xi$ is a
$d$--th root of unity.
In this way $\CC^2$ appears as the universal abelian cover $Y$ of $X$ with $H=\ZZ_d$,
and the construction in~\ref{ss:uac} applies. In particular, $\cO_{Y,0}$ has an eigenspace decomposition
$\oplus_h (\cO_{Y,0})_h$, where $(\cO_{Y,0})_h=H^0(\tilde X, \cO_{\tilde X}(-r_h))$ is the $\theta(h)$--eigenspace.
Any $f\in (\cO_{Y,0})_h$ gives two objects: firstly, $\{f=0\}$ is a curve germ in $(\CC^2,0)$, however $c(\{f=0\})$ is an effective  Weil
divisor $(C,0)$ of $(X,0)$.
In this way one realizes all the effective Weil divisors in~$(X,0)=(\CC^2/\ZZ_d,0)$.

Next, using the duality of finite maps, with the notations $\omega_{\tilde Y}=\cO_{\tilde Y}(K_{\tilde Y})$ and
$\omega _{\tilde X}=\cO_{\tilde X}(K_{\tilde X})$ we have
$\widetilde{c}_*\omega_{\tilde Y}={\mathcal Hom}_{\cO_{\tilde X}}(\widetilde {c}_*\cO_{\tilde Y},\omega_{\tilde X})=
{\mathcal Hom}_{\cO_{\tilde X}}(\oplus _h \cO_{\tilde X}(-r_h),\omega_{\tilde X})$ (cf.~\cite[Lemma 1.5]{EF} and its proof).
Since $X$ is $\QQ$--Gorenstein, $\omega_{\tilde X}=\cO_{\tilde X}(-Z_K)$, hence  $(\Omega^2_Y)_0=H^0(\tilde Y, \omega_{\tilde Y})$
has an $H$--action and an eigenspace decomposition $\oplus _h H^0(\tilde X, \cO_{\tilde X}(r_h-Z_K))$,
where $H^0(\tilde X, \cO_{\tilde X}(r_h-Z_K))$ is the $\theta(h-[Z_K])$--eigenspace.
The Gorenstein form on $\tilde Y$, $\pi^*_Y(dx \wedge dy)$, is an eigenvector in
$H^0(\tilde X, \cO_{\tilde X}(-Z_K))=H^0(\tilde X, \omega_{\tilde X})$.
Since $\xi*(dx\wedge  dy)=\xi^{1+q}dx\wedge dy$, one obtains that $[-Z_K]\in L'/L=H$ is in
 fact $(1+q)\ ({\rm mod}\ d)$.

Now, let us fix $h\in H$,  and $f\in (\cO_Y)_h$, or equivalently an effective  Weil divisor $C$.
We are searching for the subspace $M_f$ of sections $g\in (\cO_Y)_{h'}=H^0(\tilde X, \cO_{\tilde X}(-r_{h'}))$
such that $\pi^*_Y(g\frac{dx\wedge dy}{f})$, interpreted in the corresponding eigenspaces, is $H$--invariant
and it can be extended holomorphically over the generic points of each $E_v$.
This imposes two conditions, namely, $h'=h+[Z_K]$ and ${\rm div}_E(g)\geq Z_K+\ell_C'$,
where $\pi^*C=\ell_C'+\widetilde {C}$. In~\cite[Def. 3.5]{ji-JJ-numerical} (see also~\cite{ji-tesis} and \cite{CM})
the following invariant was defined associated with this context
\begin{equation}
\label{eq:kappa}
\kappa_X(C):=\dim_\CC\frac{(\cO_Y)_{h'}}{M_f}.
\end{equation}
It is shown that $\kappa_X(C)$, defined in this way, is independent of the resolution $\pi$ of $X$
(see~\cite[Prop. 2.6]{ji-JJ-numerical}). In a subsequent paper the following properties were proved:
\begin{thm}[\cite{CM}]
\label{thm:kappa}
If $(X,0)$ is a cyclic quotient surface singularity and $(C,0)$ is a reduced curve of it, then
 $\kappa_X(C)=\delta(C)$. Furthermore, in the special case when
 $\pi^*C=\widetilde C+s_h$, then
 $\kappa_X(C)=r-1$, where  $r$ is the number of irreducible components of $C$.
\end{thm}

The main purpose of the upcoming sections is to discuss possible generalizations of the definition~\eqref{eq:kappa} of
the $\kappa$--invariant to curves in \RHS\ surface singularities, their interplay, as well as their effect on general versions of
Theorem~\ref{thm:kappa}.  Note that by the above discussion the definition of $\kappa_X(C)$ in fact reads as
\begin{equation}
\label{eq:dim}
\dim_\CC \frac{H^0(\tilde X, \cO_{\tilde X}(-r_{[Z_K+\ell'_C]}))}{H^0(\tilde X, \cO_{\tilde X}(-Z_K-\ell'_C))},
\end{equation}
which suggests a possible generalization (cf.~next subsection).

\subsection{\texorpdfstring{$\kappa$}{kappa}--invariant for a surface
singularity with \RHS\  link}
\label{sec:kappa}
In this section we follow the notation from section~\ref{s:Pseries} and~\ref{sec:kappa-delta}.
Motivated by the cyclic quotient singularity case (and equation~\eqref{eq:hdim}), the right candidate for
the $\kappa$--invariant associated with an
exceptional cycle $\ell'\in L'$ in the general \RHS\ surface singularity case is the following.

\begin{dfn}
For any $\ell'\in L'$, we define
\begin{equation}\label{kappa_big}
\kappa_X(\ell'):=\mathfrak{h}(Z_K+\ell')=\dim_\CC \Big(\frac{\cO_{Y,0}}{\mathcal{F}(Z_K+\ell')}\Big)_{[Z_K+\ell']}.
\end{equation}
\end{dfn}
Note that $\kappa_X(\ell')$ in principle depends on the resolution $\tilde X$. Our purpose now is to study
the behavior of $\kappa_X(\ell')$ with the aim of defining an invariant of $(C,0)$ on~$(X,0)$.

\begin{lemma}
Let $(C,0)$ be a reduced curve in $(X,0)$, then
\begin{equation}\label{eq:kappanew}
\array{rcl}
\kappa_X(\ell_C')&=&\chi(-\ell_C')-\chi(r_{[Z_K+\ell_C']})+h^1(\cO_{\tilde X}(-r_{[Z_K+\ell_C']}))\\
&=&\chi(-\ell_C')-\chi(s_{[Z_K+\ell_C']})+h^1(\cO_{\tilde X}(-s_{[Z_K+\ell_C']})).
\endarray
\end{equation}
\end{lemma}

\begin{proof}
By~\eqref{kappa_big}, $\kappa_X(\ell'_C)=\mathfrak{h}(Z_K+\ell_C')$. Also, note that $\ell_C'\in {\mathcal S}'$,
hence $h^1(\cO_{\tilde X}(-Z_K-\ell_C'))=0$ by Grauert--Riemenschneider vanishing theorem~\ref{thm:GR}.
Furthermore, $\chi(Z_K+\ell_C')=\chi(-\ell_C')$, hence~\eqref{eq:hell} combined with~\eqref{eq:rhsh}
gives the result.
\end{proof}

\begin{cor}\label{thm:kindependent}
If $(C,0)$ is a reduced curve germ in a surface singularity $(X,0)$ with  \RHS\ link, then $\kappa_X(\ell'_C)$
does not depend on the chosen good resolution $\pi$ of~$(X,C)$.
\end{cor}
\begin{proof}
Consider
 the three terms from the right hand side of the first identity of (\ref{eq:kappanew}).  The term
 $h^1(\cO_{\tilde X}(-r_h))$ is the equivariant geometric genus of $(X,0)$ (cf. \cite{Nem-GR,Nem-CLB}), a resolution
independent object. Next, we show that $\chi(r_h)$ is also resolution independent. If $\pi$ is a  blow up of  a point on $E$ and  $E_{new}$
is  the new exceptional curve,
then in the new resolution graph $\Gamma'$ one has $Z_K'=\pi^*(Z_K)-E_{new}$, and $r_h'$ is either $\pi^*(r_h)$ or
$\pi^*(r_h)-E_{new}$. Then by a computation $\chi_{\Gamma'}(r_h')=\chi_\Gamma(r_h)$.
A similar proof shows that $\chi(-\ell_C')$ is also stable (here the assumption that $(C,0)$ is reduced is necessary, and one needs to separate the cases when the center of $\pi$ is contained or not
in $E\cap\widetilde{C}$).
\end{proof}

In the above statement  the fact that $(C,0)$ is reduced is necessary:
see  e.g. a multiple line in $(\CC^2,0)$.
This justifies the following definition of the $\kappa$--invariant of a reduced curve germ $C\subset X$ in a
\RHS\ surface singularity (extending~\eqref{eq:kappa} for quotient singularities):

\begin{dfn}
\label{def:kappa}
Let $(X,0)$ be a \RHS\ surface singularity and $(C,0)\subset (X,0)$ a reduced curve germ. The $\kappa$--invariant of $(C,0)$ in $(X,0)$
is defined as
\begin{equation}
\kappa_X(C):=\kappa_X(\ell'_C)=\mathfrak{h}(Z_K+\ell_C').
\end{equation}
\end{dfn}

The terms $\chi(-\ell_C')$ and $\chi(r_{[Z_K+\ell_C']})$ in~\eqref{eq:kappanew} are embedded  topological, while
$h^1(\cO_{\tilde X}(-r_{[Z_K+\ell_C']}))$  depends   on the homological embedding
of $(C,0)$ {\it and the analytic type of $(X,0)$}. Hence, {\it  once $(X,0)$ is fixed},
$\kappa_X(\ell_C')$ depends only on the topological embedding  $(C,0)\subset(X,0)$.

\subsection{Discussion on the $\kappa$--invariant of Cartier divisors}
\label{ex:ratnew}
Let $(X,0)$ be any normal surface singularity (where the link is not necessarily \RHS).
We assume that $(C,0)$ is reduced nonzero Cartier divisor of $(X,0)$, hence $[\ell_C']=0$ in $H$.
Then we claim that
$$\chi(-\ell'_C)=\delta(C) \ \ \mbox{(whenever $C$ is Cartier)}.$$
This can be verified in several ways. E.g., similarly as in~\ref{ss:planecurve}, if $C$ is cut
out by the holomorphic function $f$, then using A'Campo's theorem $\chi(-\ell'_C)=(r-1+\mu(f))/2$,
where, as above, $r$ is the number of irreducible components of $C=\{f=0\}$,
and $\mu(f)$ is the
Milnor number. Furthermore, by~\cite{BG80}, $(r-1+\mu(f))/2=\delta(C)$.
Note that both steps, A'Campo's and Buchweitz--Greuel's theorems, need the fact that $C$ is Cartier.

An alternative, sheaf--theoretical, proof runs as follows: using the sequence
$$0\to \cO_{{\tilde X}}(-\widetilde{C})\to \cO_{{\tilde X}}\to \cO_{\widetilde{C}}\to 0$$ one shows that
$\delta(C)=h^1(\cO_{{\tilde X}}(-\widetilde{C}))-p_g(X)$, while using
$0\to \cO_{{\tilde X}}\to \cO_{{\tilde X}}(\ell'_C)\to \cO_{\ell'_C}(\ell'_C)\to 0$ one shows that
$\chi(-\ell'_C)=h^1(\cO_{{\tilde X}}  (\ell'_C))-p_g(X)$
(for details see the proof of \eqref{eq:MAIN1}--\eqref{eq:MAIN2} below). Then
$\cO_{{\tilde X}}(-\widetilde{C})\simeq\cO_{{\tilde X}}(\ell'_C)$ whenever $C$ is Cartier.
(Here, since all the entries of $E^*_v$'s are positive, $\ell_C'$ is effective.)

Finally, the identity follows from Theorem \ref{thm:INTR4} too, which says that $A_{X,0}(C)=
\chi(-\ell_C') -\delta(C)$
depends only on the class of $C$ in ${\rm Weil}(X)/{\rm Cartier}(X)$, see also Example \ref{ex:Bl}.

In addition, if $X$ is rational, then $h^1(\cO_{{\tilde X}}(-s_{[Z_K+\ell'_C]}))=0$
(by Lipman's vanishing~\ref{thm:Lipman}), and
$\chi(s_{[Z_K+\ell'_C]})=\chi(s_{-[\ell_C']})=\chi(s_0)=\chi(0)=0$ (see Theorem~\ref{thm:kappaCoh}).
Hence~\eqref{eq:kappanew} implies
$$\kappa_X(C)=\chi(-\ell'_C)=\delta(C)\ \ \mbox{(whenever $X$ is rational and $C$ is Cartier)}.$$
The validity of such an identity $\kappa_X(C)=\delta(C)$ also shows that $\delta(C)$ depends only on
the topological position of $\widetilde {C}$ in $\tilde X$, that is, it only depends on the number
of irreducible components of $\widetilde {C}$ intersected by each $E_v$, whereas the analytic position of
the components of $\widetilde {C}$ is not essential. In particular, $\delta(C)$ can be read from (any)
embedded resolution graph of the pair~$(X,C)$.

\section{The Main Theorem}
\label{sec:mainthm}
For further references let us specify the embedded topological description of~$\kappa_X(C)$
in the rational case.
\begin{lemma}\label{eq:kappanew2}
Assume that $(X,0)$ is rational and $(C,0)$ is a reduced curve germ on it (not necessarily Cartier).
Then
$$\kappa_X(C)=\chi(-\ell_C')-\chi(s_{[Z_K+\ell_C']}).$$
\end{lemma}
\begin{proof} Use~\eqref{eq:kappanew} and  Lipman's vanishing~\ref{thm:Lipman}.
\end{proof}

Next we prove the identity $\delta(C)=\kappa_X(C)$ for $(X,0)$ is
rational and $(C,0)$ is reduced. Additionally, we will identify the obstruction
to this identity in the non--rational case.

\begin{thm}\label{thm:kappaCoh}
Let $(C,0)$ be a reduced curve germ on $(X,0)$, $\pi^*C=\widetilde{C}+\ell'_C$ and $h:=[\ell'_C]$.
\begin{enumerate}
 \item \label{thm:kappaCoh1}
 If $(X,0)$ is a normal surface singularity with \RHS\  link, then
\begin{equation}\label{eq:MAIN1}
\chi(-\ell'_C)-\chi(s_{-h})=h^1({\tilde X}, \cO_{{\tilde X}}(\ell'_C))-
h^1({\tilde X}, \cO_{{\tilde X}}(-s_{-h})).\end{equation}
 \item \label{thm:kappaCoh2}
 If $(X,0)$ is a normal surface singularity (with a non--necessarily \RHS\  link), then
\begin{equation}\label{eq:MAIN2}
\delta(C)=h^1({\tilde X}, \cO_{{\tilde X}}(-\widetilde{C}))-p_g(X,0).
\end{equation}
 \item \label{thm:kappaCoh3}
 If $(X,0)$ is a rational singularity, then
\begin{enumerate}
 \item\label{thm:kappaCoh3a}
 $p_g(X,0)=0$, $h^1({\tilde X}, \cO_{{\tilde X}}(-s_{-h}))=0$,
 \item\label{thm:kappaCoh3b}
 $\chi(s_{-h})= \chi(s_{[Z_K]+h})$,
 \item\label{thm:kappaCoh3c}
 $h^1({\tilde X}, \cO_{{\tilde X}}(\ell'_C))=h^1({\tilde X}, \cO_{{\tilde X}}(-\widetilde{C}))$.
\end{enumerate}
\end{enumerate}
In particular (using Lemma  \ref{eq:kappanew2} as well)
\begin{equation}\label{eq:MAIN4}
\kappa_X(C)=\delta(C)=\chi(-\ell'_C)-\chi(s_{[Z_K]+h})=\chi(-\ell'_C)-\chi(s_{-h}).
\end{equation}
\end{thm}
\begin{proof}
To prove~\eqref{thm:kappaCoh1} consider $\ell'=-\ell_C'$ and note that $s(\ell')=s_{-h}$ (see~\ref{ss:GLA}),
then this part is a consequence of~\eqref{eq:sell'}.
For part~\eqref{thm:kappaCoh2}, consider the short exact sequence
of sheaves $0\to \cO_{{\tilde X}}(-\widetilde{C})\to \cO_{{\tilde X}}\to \cO_{\widetilde{C}}\to 0$. Since $\widetilde {C}$ is Stein, $H^1(\cO_{\widetilde{C}})=0$, hence we have
the exact sequence $$H^0(\cO_{\tilde X})\stackrel{\nu}{\longrightarrow}
 H^0(\cO_{\widetilde{C}})\to H^1(\cO_{{\tilde X}}(-\widetilde{C}))
\to H^1(\cO_{{\tilde X}})\to 0.$$
Now $H^0({\tilde X},\cO_{\tilde X})=\cO_{X,0}$ and $\nu$ factorizes as $\cO_{X,0}\stackrel{q}
\to \cO_{C,0}\stackrel{n}{\to}
\cO_{\widetilde{C}}$, where $q$ is onto and $n$ is the normalization, hence $\dim\, {\rm coker} (\nu)=\delta(C)$.

Part~\eqref{thm:kappaCoh3a} follows from the definition of rationality and Lipman's vanishing~\ref{thm:Lipman}.
Part~\eqref{thm:kappaCoh3c} follows from the fact that
$\ell'_C+\widetilde{C}$ is numerically trivial, hence by \ref{ss:LDCG} $\cO_{{\tilde X}}(\ell'_C+\widetilde{C}) $
is an analytically trivial line bundle. Finally we prove~\eqref{thm:kappaCoh3b}.

Consider the cycle $\ell':=Z_K-s_{-h}-\ell$ with $\ell\in L$. Then $[\ell']=[Z_K]+h$.
Take the $E$--coefficients of $\ell$ sufficiently  large so that $\ell'\in L'_{\leq 0}$.
In this case (see~\ref{ss:GLA}) $s(\ell')=s_{[Z_K]+h}$. Therefore~\eqref{eq:sell'} gives
$$h^1(\cO_{{\tilde X}}(-Z_K+s_{-h}+\ell))-\chi(Z_K-s_{-h}-\ell)= h^1(\cO_{{\tilde X}}(-s_{[Z_K]+h})-\chi(s_{[Z_K]+h}).$$
By Lipman's vanishing~\ref{thm:Lipman} this transforms into 
\begin{equation}\label{eq:MAIN6}
h^1(\cO_{{\tilde X}}(-Z_K+s_{-h}+\ell))-\chi(s_{-h}+\ell)= -\chi(s_{[Z_K]+h}).
\end{equation}
Next consider the exact sequence
$$0\to \cO_{{\tilde X}}(-Z_K+s_{-h})\to
\cO_{{\tilde X}}(-Z_K+s_{-h}+\ell)\to
\cO_{\ell}(-Z_K+s_{-h}+\ell)\to 0.$$
Then by Serre duality, the formal function theorem (see~\cite{GriffithsHarris})
(for $\ell\gg 0$), and Lipman's vanishing
$$
h^0(\cO_{\ell}(-Z_K+s_{-h}+\ell)))=h^1(\cO_{\ell}(Z_K-s_{-h}+K_{\tilde X}))=h^1(\cO_{{\tilde X}}(Z_K-s_{-h}+K_{\tilde X}))=0.
$$
Hence
\begin{equation}\label{eq:MAIN7}
\begin{aligned}
h^1(\cO_{{\tilde X}}(-Z_K+s_{-h}+\ell))&=h^1(\cO_{{\tilde X}}(-Z_K+s_{-h}))+
h^1(\cO_{\ell}(-Z_K+s_{-h}+\ell))\\
&=h^1(\cO_{{\tilde X}}(-Z_K+s_{-h}))-
\chi(\cO_{\ell}(-Z_K+s_{-h}+\ell))
\end{aligned}
\end{equation}
Then~\eqref{eq:MAIN6} and~\eqref{eq:MAIN7} combined (and
$\chi(\cO_{\ell}(\tilde{\ell}))=\chi(\ell)+(\ell,\tilde{\ell})$) give
\begin{equation}\label{eq:MAIN8}
h^1(\cO_{{\tilde X}}(-Z_K+s_{-h}))-\chi(s_{-h})= -\chi(s_{[Z_K]+h}).
\end{equation}
In particular, what remains to verify  is  the following vanishing statement.
\begin{prop}\label{prop:VAN} If $(X,0)$ is rational then for any $h\in H$ one has
$$h^1(\cO_{{\tilde X}}(-Z_K+s_h))=0.$$
\end{prop}
\noindent
This, by the formal function theorem, is equivalent to the vanishing
$h^1(\cO_{\ell}(-Z_K+s_{h}))=0$ for $\ell\in L$ and $\ell\gg 0$.
This by Serre duality is equivalent to
$h^0(\cO_{\ell}(\ell-s_h))=0$ (since $\cO_{{\tilde X}}(Z_K+K_{\tilde X})$ is trivial for $X$ rational),
which will be shown next.

First notice that by the generalized Grauert--Riemenschneider vanishing theorem~\ref{thm:GGR}
and Serre duality
\begin{equation}\label{eq:MAIN9}
h^0(\cO_{\ell}(\ell-r_h))=h^1(\cO_{\ell}(-Z_K+r_{-h}))=0.
\end{equation}
Then, consider the diagram
$$\begin{array}{ccccccccc}
0 & \to & H^0(\cO_{{\tilde X}}(-s_h)) & \stackrel{\gamma}{\longrightarrow} &
H^0(\cO_{{\tilde X}}(\ell-s_h)) & \to &
 H^0(\cO_{\ell} (\ell-s_h)) &
\to & H^1(\cO_{{\tilde X}}(-s_h))\\
 & &
 \Big\downarrow\vcenter{%
 \rlap{$\scriptstyle{\alpha}$}}   & &   \Big\downarrow\vcenter{%
 \rlap{$\scriptstyle{\beta}$}} & & & & \\
0 & \to & H^0(\cO_{{\tilde X}}(-r_h)) & \stackrel{\omega}{\longrightarrow} &
 H^0(\cO_{{\tilde X}}(\ell-r_h)) & \to &
H^0( \cO_{\ell} (\ell-r_h)) & &
 \end{array} $$
By~\eqref{eq:MAIN9} $\omega$ is an isomorphism.
Since $s_h=s(r_h)$ (see~\ref{ss:GLA}) and  $\cF(\ell')=\cF(s(\ell'))$ (cf.~\cite{NemOSZ}),
$\alpha$ is also an isomorphism. In particular, $\beta\circ \gamma$ is an isomorphism.
But both $\beta$ and $\gamma$ are injective (inclusions), hence both should be
isomorphisms. Finally notice that
$H^1(\cO_{{\tilde X}}(-s_h))=0$ by Lipman's vanishing. Hence $H^0(\cO_{\ell}(\ell-s_h))=0$.
\end{proof}

\begin{rem}
In the formulation of Theorem~\ref{thm:kappaCoh}\eqref{thm:kappaCoh3} we listed those properties,
which (together with Lipman's vanishing) basically obstruct the identity $\kappa_X(C)=\delta(C)$.
In fact, \eqref{thm:kappaCoh3a} is
exactly  the rationality of $(X,0)$. Regarding  the other two properties the following examples show that they
might fail too if $(X,0)$ is not rational.
\end{rem}

\begin{exam}\label{ex:DOESNOT}
Though in the study of normal surface singularities the term
$\chi(s_h)$ appears rather frequently, till this work the authors were not aware
of the topological  identity $\chi(s_{-h})=\chi(s_{[Z_K]+h})$, valid
for rational  singularities.
We wish to emphasize that this identity is not true in general. Take e.g.~the star--shaped graph
with central vertex $E_0$ decorated by $-1$, and with four legs, each
of length one, decorated by $-4$, $-4$, $-4$ and $-10$.
Then, $  Z_K=(26/3,8/3,8/3,8/3,5/3)$,  $s_{[Z_K]}=(8/3,2/3,2/3,2/3,2/3)$ and
$ \chi(s_{[Z_K]})=-2$. Hence $\chi(s_{-h})=\chi(s_{[Z_K]+h})$ might fail even for $h=0$.
\end{exam}

\begin{exam}\label{ex:NON}
For  non-rational singularities $(X,0)$, even if $H=0$, we cannot expect the  identity
$\kappa_X(C)=\delta(C)$  in general.
Indeed, take e.g. the non-rational (minimally elliptic)
singularity $X=\{x^2+y^3+z^7=0\}$ and the  two topologically equivalent  reduced curve germs on $X$:
$C$ is the Cartier divisor cut out by $z=0$, and $\widetilde{C}_1$ any generic smooth transversal
curvette supported by the same irreducible exceptional divisor as $\widetilde{C}$ on the minimal
good resolution of $X$. Then $C$ is the plane cusp $\{z=x^2+y^3=0\}$ with $\delta(C)=1$, while
$C_1$ is smooth (see e.g.~\cite[Example 9.4.3]{Nem-PS}) with $\delta(C_1)=0$.
On the other hand, since the embedded topological types coincide  $\kappa_X(C)=\kappa_X(C_1)$.
In fact, by~\eqref{eq:kappanew} and Example~\ref{ex:ratnew} applied for Cartier divisor $C$ we have
$\kappa_X(C)= \chi(-\ell_C')+h^1(\cO_{{\tilde X}})=\delta(C)+p_g(X)=1+1=2$.
Let us analyze this example from the point of view of (c). Since $C$ is Cartier, (c)  is true
for $C$, and by~\eqref{eq:MAIN2}
$h^1({\tilde X}, \cO_{{\tilde X}}(\ell'_C))=h^1({\tilde X}, \cO_{{\tilde X}}(-\widetilde{C}))=2$.
Since $\ell_C'=\ell_{C_1}'$ we get that
$h^1({\tilde X}, \cO_{{\tilde X}}(\ell'_{C_1}))=2$ too. But, again by~\eqref{eq:MAIN2}
(and the above discussion) $h^1({\tilde X}, \cO_{{\tilde X}}(-\widetilde{C}_1))=1$.
Hence (c) is not true for (the non--Cartier)~$C_1$.
\end{exam}

\section{More examples and some delta invariant formulae}\label{s:delta}

In this section we review some facts  about the delta invariant of a curve germ,
which are relevant from the point of view of the present note. In some parts
we  follow \cite{BG80}
and \cite{StevensThesis}.

Let $(C,0)$ be the germ of a complex reduced curve singularity, let $n:(C,0)^{\widetilde{}}\to(C,0)$
be the normalization, where $(C,0)^{\widetilde{}}$ is the multigerm $(\widetilde{C},n^{-1}(0))$.
The delta invariant is defined as $\dim_{{\CC}}n_*\cO_{(C,0)^{\widetilde{}}}/\cO_{(C,0)}$.
We also write $r$ for the number of irreducible components of $(C,0)$, as usual.
The delta invariant of a reduced curve can be determined inductively from the delta invariant of
the components and the \emph{Hironaka generalized intersection multiplicity}.
Indeed, assume $(C,0)$ is embedded in some $(\CC^n,0)$, and assume that $(C,0)$ is the union of
two (not necessarily irreducible) germs $(C_1',0)$ and $(C_2',0)$ without common irreducible components.
Assume that $(C_i',0)$ is defined by the ideal $I_i$ in $\cO_{(\CC^n,0)}$ ($i=1,2$).
Then one can define \emph{Hironaka's intersection multiplicity} by $(C_1',C_2')_{Hir}:=\dim _{\CC}\, (\cO_{(\CC^n,0)}/ I_1+I_2)$.
Then, one has the following formula of Hironaka, see~\cite{Hironaka} or~\cite[2.1]{StevensThesis} and~\cite{BG80},
\begin{equation}\label{eq:delta11}
\delta(C,0)=\delta(C_1',0)+\delta(C_2',0)+(C_1',C_2')_{Hir}.
\end{equation}
In particular, if $(C,0)$ has irreducible decomposition $\cup_{i=1}^r(C_i,0)$, and we set
$(C^j,0):=\cup_{i=j+1}^r(C_i,0)$ then by (\ref{eq:delta11}) inductively (cf. \cite{Hironaka})
\begin{equation}\label{eq:delta12}
\delta(C,0)=\sum_{i=1}^r \delta(C_i,0)+ \sum_{i=1}^{r-1} (C^i,C_i)_{Hir}.\end{equation}

\begin{exam}\label{ex:delta22} \cite{BG80,StevensThesis}
Assume that $(C,0)$ is (analytically equivalent with) the union of the coordinate
axes of $(\CC^r,0)$ (called {\it ordinary $r$--tuple}). Then using any of the above formulae  we get that $\delta(C,0)=r-1$. Furthermore, for any $(C,0)$ (since $(C_1',C_2')_{Hir}\geq 1$ always) we have $\delta(C,0)\geq r-1$. Conversely, if $\delta(C,0)= r-1$ (is the sharp minimum)
then $(C,0)$ is necessarily an
ordinary $r$--tuple.
\end{exam}

\begin{exam}\cite[3.5]{StevensThesis}
Let us fix a surface singularity $(X,0)$.
We say that $(X,0)$ is a Kulikov singularity, cf. \cite{Karras80},
if  there exists a resolution $\tilde X$, in which  the fundamental cycle $Z_{min}\in L$
(the unique minimal cycle of ${\mathcal S}\setminus \{0\}$, cf. \cite{Artin66})
 satisfies the following property:
if $(Z_{min}, E_v)<0$, then the $E_v$--multiplicity of $Z_{min}$ is one.
Assume additionally that $(X,0)$ is rational. Then, by \cite{Artin66},
 the multiplicity $r$ of $(X,0)$ is
$-Z_{min}^2$ and the embedded dimension
of $(X,0)$ is $r+1$. Let $f$ be the generic linear function of $(\CC^{r+1},0)$, it induces the
`generic linear section'  of $(X,0)$. Let $(C,0):=\{f=0\}\cap  (X,0)\subset (X,0)$.
Then $(C,0)$ is embedded
in the linear hyperplane $\{f=0\}$ of $(\CC^{r+1},0)$, hence in some $\CC^r$. Furthermore,
since the maximal ideal of $\cO_{(X,0)}$ has no base point, $(C,0)$ is reduced,
and
$r=-Z_{min}^2=(\widetilde{C},Z_{min})=(\widetilde{C},E)$
shows that in fact $(C,0)$ has $r$ irreducible components. The Kulikov property
(the function $f$ intersects each components with multiplicity one) guarantees that each
irreducible component of $(C,0)$ is smooth. Again, by base point freeness, all the $r$ smooth components of $(C,0)$ in $\{f=0\}=(\CC^r,0)$ are in general position, hence
$(C,0)$ in fact is an ordinary $r$--tuple. (For more see e.g. \cite{StevensThesis}.)

Therefore, we obtain, that the generic linear section $(C,0)$ of a rational
Kulikov  singularity is an ordinary $r$--tuple, with $\delta (C,0)=r-1$, where
$r$ is the multiplicity of $(X,0)$.
\end{exam}
\begin{exam}\label{ex:delta33}
Consider again a rational Kulikov singularity, and set $r_v:=-(Z_{min},E_v)$.
Then $r=\sum _vr_v$, and the strict transform $\widetilde {C}$ of the
generic linear section $(C,0)$ intersects each $E_v$
along  $r_v$ transversal components.
Consider any sub--collection $(C',0)$ of these components with total number  $r'$.
 Since $(C,0)$ is an ordinary
$r$--tuple, $(C',0)$ is an ordinary $r'$--tuple too. In particular, $\delta(C',0)=r'-1$.

Set any collection $\{r_v'\}_v$ with $0\leq r_v'\leq r_v$, and consider a curve germ $(C',0)
\subset (X,0)$ with $r'=\sum_vr_v'$ components,
such that its strict transform $\widetilde {C}'$
intersects $E_v$ transversally via $r_v'$
components. We claim that  $\delta(C',0)=r'-1$ again.
Indeed, complete $(C',0)$ to a germ $(C,0)$ with $r$ components, such that its strict transform
$\widetilde {C}$ intersects each $E_v$ with $r_v$ components, and each intersection is smooth.
Then $\widetilde{C}+Z_{min}$ is a numerically trivial divisor in $\tilde X$, hence, by
(\ref{eq:picard}) it is a principal divisor ${\rm div}(\widetilde{f})$ cut out by a holomorphic function $\widetilde{f}$ of $\tilde X$. Since all the coefficients of $Z_{min}$ are positive,
$\widetilde{f}$ vanishes along $E$, hence contracts to a continuous function $f$ of $X$, this is
analytic by the normality of $X$. Hence, $(C,0)=\{f=0\}$ is in fact cut out by a function whose
total transform along $E$ is $Z_{min}$, $(C',0)$ is a sub--collection of it,
and all the above arguments apply.
\end{exam}
\begin{exam}
Finally, let us list some well
known rational Kulikov singularities: cyclic quotient (with its minimal resolution), or any
rational singularity  when the fundamental cycle is reduced.
In particular,  for all such singularities the above discussion
applies.
\end{exam}
\begin{exam}\label{ex:2proof}
 Assume that $X$ is a cyclic quotient singularity and $\tilde X$ is its minimal
resolution. Fix any $h\in H=\ZZ_d$ and consider $s_h=\sum r_v'E^*_v$ and $r_v=-(Z_{min},E_v)$
as above. There is an algorithm  in \cite[10.3.3]{NemOSZ}, which provides $\{r_v'\}_v$;
according to it one sees immediately that $r_v'\leq r_v$ for all $v$.
In particular, if $(C',0)\subset (X,0)$ such that $\pi^*(C')=\widetilde {C}'+s_h$ then
$\delta(C',0)= r'-1$,  with $r'=\sum_vr'_v$ the number of components of $C'$, cf. Example \ref{ex:delta33}.
(This is a new proof of the second part of Theorem \ref{thm:kappa}.)
\end{exam}
\begin{exam}
In general we cannot expect the identity $\delta(C)=r-1$ (cf. Example \ref{ex:delta22}), the curve $C$ (even if it has only smooth components)   is not necessarily ordinary $r$--tuple.
Consider e.g. the cyclic quotient singularity $X=\frac{1}{4}(1,3)$, whose resolution graph is the
 $\mathbb A_3$ graph. The action is $\xi*(x,y)=(\xi x, \xi^3 y)$, hence the invariant ring is
 generated by
 $u=x^4$, $v=y^4$ and $w=xy$. In particular, $X=\{uv=w^4\}$.
If  $f(x,y)=x^{12}-y^{4}$, then $f$ is invariant, hence the corresponding divisor
$C=c(\{f=0\})$ is Cartier.
It is given by $u^3=v$ on $X$. Therefore, $C$ is $\{uv-w^4=v-u^3=0\}$, isomorphic with the plane curve singularity $\{w^4=u^4\}$ with $r=4$ and $\delta(C)=6$. (Clearly, $h=0$ and $\ell'_C\not=s_h$.)
\end{exam}
\begin{exam}\label{ex:Bl}
It is well  known that for plane curve singularities $(C_1,C_2)_{Hir}$ coincides with the
local intersection multiplicity $(C_1,C_2)_{\CC^2,0}$ and
\begin{equation}\label{eq:deltaplanecurves}
\delta (C)=\sum_i\delta(C_i)+\sum_{i<j} (C_i,C_j)_{\CC^2,0}.\end{equation}
This can be deduced from Hironaka's formula (\ref{eq:delta12}) as well, since for plane curves
$(C',C_1\cup C_2)_{\CC^2,0}=(C', C_1)_{\CC^2,0}+(C',C_2)_{\CC^2,0}$, a property which usually
fails for non--plane germs in the general context  of Hironaka for $(-,-)_{Hir}$.

Note that for plane curve germs one also has $(C_1,C_2)_{\CC^2,0}=-(\ell_{C_1}', \ell_{C_2}')$,
the second term computed in the lattice of a good embedded resolution of the pair
$C_1\cup C_2\subset \CC^2$.  Having in mind this formula, it is natural to
 consider the following generalization.
Let $X$ be a normal surface singularity, $C_1$, $C_2$ two Weil divisors on it without common components, and let $\tilde X$ be a good embedded resolution of the pair
$C_1\cup C_2\subset X$. Then define $(C_1,C_2)_X$ as $-(\ell_{C_1}', \ell_{C_2}')$
(in the lattice of $\tilde X$), cf. \cite{mumford,Sakai84}. If the link is an
integral homology sphere then it is an integer, however in general it is a rational number.
In particular, in general does not equal the
Hironaka pairing. Even more, $(C_1,C_2)_X$ might depend on the choice of $X$.
Take e.g. $C_1$ and $C_2$ the two components of $\{z=xy=0\}$  embedded in $X_n=\{z^n=xy\}$.
Then $(C_1,C_2)_{X_n}=1/n$.

If  $C_1, C_2$ are two effective Weil divisors on a  normal surface singularity with no common components then by Theorem \ref{thm:INTR4}
$A_{X,0}(C_1+C_2)=\chi(-\ell'_{C_1}-\ell'_{C_2})-\delta(C_1\cup C_2)$ (and similarly for $C_1$ and $C_2$).
Since $\chi(-\ell'_{C_1}-\ell'_{C_2})=\chi(-\ell'_{C_1})+\chi(-\ell'_{C_2})-(\ell'_{C_1},\ell'_{C_2})$ and $\delta(C_1\cup C_2)=\delta(C_1)+\delta(C_2)+(C_1,C_2)_{Hir}$
one obtains the following statement.

\begin{prop}\label{prop:UJUJ}
\begin{equation}
\label{eq:multiplicity}
(C_1,C_2)_{Hir}=(C_1,C_2)_X+A_{X,0}(C_1)+A_{X,0}(C_2)-A_{X,0}(C_1+C_2).
\end{equation}\end{prop}
In particular,  if one of the divisors, say $C_1$, is Cartier,
then from Theorem \ref{thm:INTR4} $A_{X,0}(C_2)=A_{X,0}(C_1+C_2)$ and $A_{X,0}(C_1)=0$,
 which imply
$$(C_1,C_2)_{Hir}=(C_1,C_2)_X.$$
\end{exam}

\begin{exam}
Finally we wish to emphasize that all terms $s_h$, $\chi(s_h)$, $\chi(-\ell_C')$ can be very arithmetical.
In order to be more explicit, we will give another new proof (based on formulae from~\cite{NemOSZ})
of the fact that if $X$ is the cyclic quotient singularity $\frac{1}{d}(1,q)$, $\tilde X$ is its
minimal resolution, and $\ell_C'=s_h$ for some $h\in H$ ($h\not=0$), then $\delta(C)=r-1$
(cf. Theorem~\ref{thm:kappa} and Example~\ref{ex:2proof}).
This reads as follows: write $s_h=\sum_v r_vE^*_v$. Then we claim that
$$Exp:=\chi(-s_h)-\chi(s_{-h}) \ \ \mbox{equals $-1+\sum_vr_v=-1+r$}.$$
Since $\chi(-s_h)=\chi(s_h)-(s_h,Z_K)$, the expression $Exp$ transforms into
$$Exp=\chi(s_h)-\chi(s_{-h})-(s_h,Z_K).$$
Let $E_1,\ldots, E_s$  be the irreducible exceptional divisors (in this order on the bamboo $\Gamma$).
Then $[E^*_s]$ generates $H=\ZZ_d$. We set $h=[aE^*_s]$ for some $0<a<d$. We also
set $0<q'<d$ so that $qq'\equiv 1$ (mod $d$).
In the sequel we use $\{-\}$ for the fractional part, and $\lfloor -\rfloor$ for the integral part.
Then, by~\cite[10.5.1]{NemOSZ}
$$\chi(s_h)=\chi(s_{[aE^*_s]})=\frac{a(1-d)}{2d} + \sum_{i=1}^a\, \left\{ \frac{iq'}{d}\right\}.$$
Obviously, $-h=[(d-a)E^*_s]$, hence $\chi(s_{-h})$ equals
\begin{equation*}
\frac{(d-a)(1-d)}{2d} + \sum_{j=1}^{d-a}\, \left\{ \frac{jq'}{d}\right\}
 = \frac{(d-a)(1-d)}{2d} + \sum_{i=1}^{d-1}\,\left( 1- \left\{ \frac{iq'}{d}\right\}\right)
 -\sum_{i=1}^{a-1}\,\left( 1- \left\{ \frac{iq'}{d}\right\}\right).\end{equation*}
After a computation
$$\chi(s_h)-\chi(s_{-h})=\frac{a}{d}-1 -\left\{\frac{aq'}{d}\right\}.$$
Since by~\eqref{eq:ZK} $Z_K=E-E_1^*-E_s^*$, we get
\begin{equation}\label{eq:FINAL1}
Exp=\frac{a}{d}-1 -\left\{\frac{aq'}{d}\right\}+ (E_1^*+E_s^*, s_h)+\sum_i r_i.
\end{equation}
Let us denote (as in \cite{NemOSZ}) $-(E_1^*, E_i^*)$ by $n_{i+1,s}/d$ ($1\leq i\leq s$). Then
$(E_1^*,s_h)=-(\sum_i r_i\cdot n_{i+1,s})/d$. This by (1) of \cite[Lemma 10.3.1]{NemOSZ} is
exactly $-a/d$. Hence
\begin{equation}\label{eq:FINAL2}(E_1^*,s_h)=-a/d.\end{equation}

In order to compute $(E_s^*,s_h)$ we create formally the symmetric situation.
Since $[E^*_1]=[qE^*_s]$ and  $[E^*_s]=[q'E^*_1]$,
$h=[aE^*_1]=[aq'E_1^*]$. Write $a' \equiv aq'$ (mod $d$), $0<a'< d$. Then for $h=[a'E_1^*]$ and
$s_h=s_{[a'E_1^*]}=\sum_i r_iE^*_i$ the symmetric statement of~\eqref{eq:FINAL2} is
\begin{equation}\label{eq:FINAL3}(E_s^*,s_h)=-\frac{a'}{d}=\left\lfloor
 \frac{aq'}{d}\right\rfloor - \frac{aq'}{d}.\end{equation}
Then~\eqref{eq:FINAL1},\eqref{eq:FINAL2}, and~\eqref{eq:FINAL3} combined give $Exp=-1+r$.

Having in mind similar manipulations (e.g.~with Dedekind sums) the identity
$\chi(-s_h)-\chi(s_{-h})=-1+r$ can be interpreted as a \emph{reciprocity law}.
\end{exam}

\providecommand{\bysame}{\leavevmode\hbox to3em{\hrulefill}\thinspace}
\providecommand{\MR}{\relax\ifhmode\unskip\space\fi MR }
\providecommand{\MRhref}[2]{%
  \href{http://www.ams.org/mathscinet-getitem?mr=#1}{#2}
}
\providecommand{\href}[2]{#2}

\end{document}